\documentclass[a4, 12pt, 
]{amsart}

\usepackage[dvipdfmx]{hyperref}
\usepackage{amsmath}
\usepackage{amssymb}
\usepackage{hyperref}
\hypersetup{
	colorlinks=true,
	linkcolor=red,
	citecolor=blue}
	\usepackage[utf8]{inputenc}
\usepackage[T1]{fontenc}
\usepackage{color}
\usepackage{xcolor} 
\usepackage{ulem}
\usepackage{calligra}
\usepackage{mathrsfs}
\usepackage[all]{xy}

\newtheorem{theo}{Theorem}[section]
\newtheorem{prop}[theo]{Proposition}
\newtheorem{lemm}[theo]{Lemma}
\newtheorem{cor}[theo]{Corollary}

\numberwithin{equation}{section}

\theoremstyle{definition}
\newtheorem{defi}[theo]{Definition}

\theoremstyle{remark}
\newtheorem{rem}[theo]{Remark}

\newcommand{\Ker}[0]{\operatorname{Ker}}

\newcommand{\End}[0]{\operatorname{End}}
\newcommand{\Herm}[0]{\operatorname{Herm}}
\newcommand{\rank}[0]{\operatorname{rank}}

\newcommand{\Nak}{\mathrm{Nak.}}

\newcommand{\Id}{\mathrm{Id}}

\newcommand{\deldel}{\sqrt{-1}\partial \overline{\partial}}
\newcommand{\Dbar }{\overline{\partial}}
\newcommand{\e}{\varepsilon}
\newcommand{\ai}{\sqrt{-1}}
\newcommand{\R}{\mathbb{R}}

\newcommand{\C}{\mathbb{C}}

\newcommand{\Z}{\mathbb{Z}}

\newcommand{\glo}{\mathrm{glo.}}
\newcommand{\loc}{\mathrm{loc.}}

\usepackage{geometry}
\begin{document}

\title[Nakano positivity: approximations and applications]
{Nakano positivity of singular Hermitian metrics: \\ 
Approximations and applications}

\author{Takahiro Inayama$^{1}$}
\author{Shin-ichi MATSUMURA$^{2}$}

\address{$^{1}$ Department of Mathematics\\
Faculty of Science and Technology\\
Tokyo University of Science\\
2641 Yamazaki, Noda\\
Chiba, 278-8510\\
Japan
}
\email{{\tt inayama\_takahiro@rs.tus.ac.jp}}
\email{{\tt inayama570@gmail.com}}

\address{$^{2}$ Mathematical Institute, Tohoku University, 
6-3, Aramaki Aza-Aoba, Aoba-ku, Sendai 980-8578, Japan.}
\email{{\tt mshinichi-math@tohoku.ac.jp}}
\email{{\tt mshinichi0@gmail.com}}

\date{\today, version 0.01}

\renewcommand{\subjclassname}{%
\textup{2020} Mathematics Subject Classification}


\maketitle

\begin{abstract}
This paper studies the approximation of singular Hermitian metrics on vector bundles 
using smooth Hermitian metrics with Nakano semi-positive curvature on Zariski open sets. 
We show that singular Hermitian metrics capable of this approximation 
satisfy Nakano semi-positivity as defined through the $\overline{\partial} $-equation with optimal $L^2$-estimates. 
Furthermore, for a projective fibration $f \colon  X \to Y$ with a line bundle $L$ on $X$, 
we provide a specific condition under which the Narasimhan-Simha metric 
on the direct image sheaf $f_{*}\mathcal{O}_{X}(K_{X/Y}+L)$ admits this approximation. 
As an application, we establish several vanishing theorems. 
\end{abstract}

\tableofcontents

\section{Introduction}\label{Sec-1}

Generalizing the theory from smooth to singular Hermitian metrics is a significant topic in complex geometry. 
This aspect is particularly vital when addressing singular Hermitian metrics on (holomorphic) vector bundles.
A pivotal challenge lies in establishing an \textit{appropriate} definition of Nakano positivity 
and expanding the theoretical framework beyond smooth Hermitian metrics.
Primarily, there are two approaches to defining Nakano positivity: 
The first approach employs an approximation with smooth Hermitian metrics, 
initiated by de Cataldo \cite{deC98} and has been continually explored (see \cite{GMY22, GMY23} for examples).
The second approach employs the $\overline{\partial} $-equation with optimal $L^2$-estimates of H\"ormander type, 
which has recently attracted attention and is developing rapidly (see  \cite{Ber98, HI21, Ina22, DNWZ23} for examples).
Despite these advancements, defining Nakano positivity for singular Hermitian metrics remains an unresolved issue, 
with challenges still to be addressed:
\begin{enumerate}
\item[(1)] How are the first and second approaches of Nakano positivity related?
\item[(2)] What are typical examples of Nakano positive singular Hermitian metrics?
\item[(3)] Is it possible to derive vanishing theorems for singular Hermitian metrics?
\end{enumerate}

The primary purpose of this paper is to propose a new framework regarding the Nakano positivity of a singular Hermitian metric $h$ on a vector bundle $E$ that encompasses the existing theories developed by de Cataldo and Guan--Mi--Yuan.
Our definition of Nakano semi-positivity, denoted by $ \ai\Theta_{h} \geq^s_{\Nak} 0 $, 
 is based on the existence of a sequence $\{ h_{s}\}_{s=1}^{\infty}$ of singular Hermitian metrics  
converging to $h$, where each metric is smooth and has arbitrarily small Nakano negativity on a Zariski open set 
(see Definition \ref{def-singNak:sequenceNew} for the precise condition). 

Our definition builds upon the first approach, drawing inspiration from \cite{GMY22, GMY23}. 
Compared to the approach in \cite{GMY22, GMY23}, our definition is more flexible 
in that $h_{s}$ is not required to be smooth on the entire space. 
This flexibility enables us to establish Theorem \ref{thm-direct}, 
which demonstrates that the direct image sheaf $f_{*}(\mathcal{O}_{X}(K_{X/Y}+L))$  
satisfies our Nakano semi-positivity under a reasonable assumption. 
This result partially solves Problem (2) and is also important in terms of applications to algebraic geometry.

\begin{theo}\label{thm-direct}
Let $f \colon X \to Y$ be a projective fibration between $($not necessarily compact$)$ 
complex manifolds $(X, \omega_{X})$, $(Y, \omega_{Y})$ with Hermitian forms. 
Let $\theta$ be a continuous real $(1,1)$-form on $Y$ and 
$L$ be a line bundle on $X$ satisfying the following conditions$:$
\begin{itemize}

\item $L$ admits a singular Hermitian metric $h$ such that $\sqrt{-1}\Theta_{h} \geq f^{*}\theta$ holds on $Y$ and 
$\{x \in X \mid \nu(h, x)>0\}$ is not dominant over $Y$, 
where $\nu(h, x)$ denotes the Lelong number of the weight of $h$ at a point $x \in X$. 

\item $L$ is $f$-ample in the following sense: $L$ admits a smooth Hermitian metric $g$ 
such that $\sqrt{-1}\Theta_{g} + f^{*} \omega_{Y} \geq \omega_{X}$ holds on $X$. 
\end{itemize}
Then, the induced Narasimhan-Simha metric $G$ on 
the direct image sheaf $f_{*} \mathcal{O}_{X}(K_{X/Y}+L)$ 
satisfies that 
$
\ai\Theta_G \geq^{s}_\Nak  \theta \otimes \Id_{G}  \text{ on }Y 
\text{ in the sense of Definition \ref{def-singNak:sequenceNew-sheaf}.}
$
\end{theo}

In response to Problem (1), we establish Theorem \ref{thm-seqimpliesL2}, 
which demonstrates that our definition of Nakano positivity (as defined in Definition \ref{def-singNak:sequenceNew}) 
implies the Nakano positivity defined by the $\overline{\partial} $-equation (see Definition \ref{def-singNak:l2global}).

\begin{theo}\label{thm-seqimpliesL2}
Let $E$ be a vector bundle on a complex manifold $X$ and $h$ be a singular Hermitian metric on $E$. 
Let $\theta$ be a continuous real $(1,1)$-form on $X$.
If 
$
\ai\Theta_h \geq^{s}_\Nak \theta \otimes\Id_{E_h} 
\text{ holds in the sense of Definition \ref{def-singNak:sequenceNew},}
$
then 
$
\ai\Theta_h \geq^{L^2}_{\Nak} \theta\otimes\Id_{E_h} \text{ holds } 
\text{   in the sense of Definition \ref{def-singNak:l2global}.}
$
\end{theo}

As mentioned in Problem (3), an advantage of Nakano-positivity is its applicability to vanishing theorems. 
This paper establishes new vanishing theorems (see Theorems \ref{mainthm-vanishing} and \ref{thm-vanishing-vec}), 
which generalize some vanishing theorems of Nadel type and related works in \cite{Iwa21, Ina20, Ina22}.

\begin{theo}\label{mainthm-vanishing}
Let $(X,\omega)$ be a weakly pseudoconvex K\"ahler manifold and 
$(E, h)$ be a vector bundle on $X$ with a singular Hermitian metric $h$. 
Assume that 
    \[
    \ai\Theta_h \geq^s_{\Nak} \omega'\otimes \Id_{E_h} \text{ holds in the sense of Definition \ref{def-singNak:sequenceNew}}
    \]
    for some continuous Hermitian form $\omega'$ on $X$. 
    Then, it follows that 
    \[
H^q(X, K_X\otimes \mathcal{E}(h))=0
\]
for $q>0$.
Here $\mathcal{E}(h)$ denotes the sheaf of germs of locally $L^{2}$-integrable sections of $E$  with respect to $h$. 
\end{theo}
\begin{theo}\label{thm-vanishing-vec}
Let $(X,\omega)$ be a weakly pseudoconvex K\"ahler manifold and  
$(E, h)$ be a vector bundle on $X$ with a singular Hermitian metric $h$. 
Consider  the projective space bundle $f \colon \mathbb{P}(E) \to X$ 
and the hyperplane  bundle $\mathcal{O}_{\mathbb{P}(E)}(1)$ on $\mathbb{P}(E)$. 
Let $h_{q}$ be the singular Hermitian metric  on  $\mathcal{O}_{\mathbb{P}(E)}(1)$ 
induced by $h$ and $f^{*}E \to \mathcal{O}_{\mathbb{P}(E)}(1)$. 
Assume that $(\mathcal{O}_{\mathbb{P}(E)}(1), h_{q})$ is a big line bundle in the strong sense$:$  
the curvature current $\sqrt{-1}\Theta_{h_{q}}$ is a K\"ahler current and 
$\{x \in \mathbb{P}(E) \mid \nu(h_{q}, x)>0\}$ is not dominant over $X$.  
Then,  it follows that 
    \[
H^q(X, K_X\otimes S^{m}\mathcal{E}\otimes \det \mathcal{E}(S^{m}h \otimes \det h))=0
\]
for $q>0$ and $m \geq 0$.
\end{theo}

\subsection*{Acknowledgment}\label{subsec-ack}
T.\,I.\,is supported by Grant-in-Aid for Research Activity Start-up $\sharp$21K20336 and Grant-in-Aid for Early-Career Scientists $\sharp$23K12978 from Japan Society for the Promotion of Science (JSPS).
S.\,M.\,is supported 
by Grant-in-Aid for Scientific Research (B) $\sharp$21H00976 
and Fostering Joint International Research (A) $\sharp$19KK0342 from JSPS. 

\section{Definitions of Nakano positivity and preliminary results}\label{Sec-prelim}

This section is devoted to explaining the definitions of positivity for singular Hermitian metrics 
and preliminary results needed for subsequent discussions. 
The primary purpose is to introduce a new definition of Nakano positivity mentioned in Section \ref{Sec-1} (see Definition \ref{def-singNak:sequenceNew}).

Throughout this paper, 
let $X$ be a complex manifold of dimension $n$ and $E$ be a vector bundle of rank $r$ on $X$. 
Moreover, we adopt the following convention:

\begin{defi}[Associated Hermitian forms and Nakano/Griffiths positivity]
A smooth Hermitian metric $h$ on $E$ defines 
the Chern curvature and associated Hermitian form:
$$
\ai\Theta_{h} \in C^{\infty}(X, \Lambda^{1,1}\otimes \End{(E)})
\text{ and } 
\ai \widetilde \Theta_h \in C^{\infty}(X, \Herm(T_{X}\otimes E)).
$$
The metric $h$ is said to be \textit{Nakano semi-positive} 
if $\ai \widetilde \Theta_h$ is semi-positive definite for all elements $\tau\in T_X\otimes E$ 
and \textit{Griffiths semi-positive} if $\ai \widetilde \Theta_h$ is semi-positive definite 
for all decomposable elements $\xi\otimes s \in T_X\otimes E$. 

In general, for a real $(1,1)$-form $\Xi$ valued in $ \End{(E)}$, 
we can define the associated Hermitian form on $T_X\otimes E$, denoted by $\Xi_h$. 
For instance, for a real $(1,1)$-form $\theta$, 
the notation $\theta \otimes \Id_{E_{h}}$ refers to the Hermitian form associated with $\theta \otimes \Id_E$. 
These forms are often identified when there is no risk of confusion.
\end{defi}

We review the notion of Griffiths positivity for singular Hermitian metrics, 
which coincides with the definition involving curvature presented above for smooth Hermitian metrics. 

\begin{defi}\label{def-sHm}
(Singular Hermitian metrics, {\cite[Section 3]{BP08}, \cite[Definition 17.1]{HPS18}, \cite[Definition 2.2.1]{PT18}}). 
A \textit{singular Hermitian metric} $h$ on $E$ is defined as a measurable map from the base manifold $X$ 
to the space of non-negative Hermitian forms on the fibers of $E$ such that $0 < \det h < \infty$ holds almost everywhere.
\end{defi}

\begin{defi}[Griffiths positivity]
\label{def-GrifPos}
A singular Hermitian metric $h$ on $E$ is said to be {\it Griffiths semi-positive} 
if $\log |u|_{h^{*}}$ is psh (plurisubharmonic) 
for any local (holomorphic) section $u$ of $E^{*}$, 
where $h^{*}$ denotes the induced metric on the dual vector bundle $E^{*}$. 
\end{defi}

The following definition, inspired by \cite[Definition 2.4.1]{deC98} and \cite[Definition 1.1]{GMY23}, 
introduces a new approach to defining singular Nakano positivity via an approximation. 

\begin{defi}[Nakano positivity in the sense of approximations]\label{def-singNak:sequenceNew}
Let $\omega$ be a Hermitian form and $\theta$ be a continuous  real $(1,1)$-form on $X$. 
Consider a singular Hermitian metric $h$ on $E$ with $h^*$ being upper semi-continuous 
(i.e.,\,the function $|u|^2_{h^*}$ is upper semi-continuous for any local  section $u$ of $E^*$). 
We define \textit{the $\theta$-Nakano positivity in the sense of approximations}, denoted by 
\[
\ai\Theta_{h} \geq^s_{\Nak} \theta \otimes \Id_{E_h},
\]
by the existence of 
the following data: $(\{ X_j\}_{j=1}^{\infty}, \{ \Sigma_{j,s}\}_{j,s=1}^{\infty}, \{ h_{j,s}\}_{j,s=1}^{\infty})$ consisting of  
\begin{itemize}
\item $\{ X_j \}_{j=1}^{\infty}$ is an open cover of $X$ such that 
$X_j\Subset X_{j+1}\Subset X$ for any $j \in \mathbb{Z}_{+}$; 
\item $\Sigma_{j,s}\subset X_j$ is a (proper closed) analytic subset of $X_{j}$; 
\item $h_{j,s}$ is a $C^{2}$-Hermitian metric on $E|_{X_{j} \setminus \Sigma_{j,s} }$; 
\end{itemize}
and satisfying Conditions (a) and (b):
\begin{enumerate}
\item[(a)] Set $\Sigma_{j}:=\cup_{s=1}^{\infty}\Sigma_{j,s}$. 
Then, for any $x\in X_j\setminus \Sigma_{j}$ and $e\in E_x$, we have 
$$
|e|_{h_{j,s}}\nearrow |e|_{h} \text{ as } s \nearrow \infty. 
$$ 
\item[(b)] There exist continuous functions $\lambda_{j,s}$ and $\lambda_j$ on $\overline{X}_{j}$ 
such that 
\begin{itemize}
\item $0\leq \lambda_{j,s}\leq \lambda_j$ on $X_j$; 
\item  $\lambda_{j,s}\to 0$ almost everywhere on $X_j$; 
\item $\ai\Theta_{h_{j,s}}\geq_\Nak (\theta-\lambda_{j,s}\omega)\otimes \Id_{E_{h_{j,s}}}$ 
holds at any point $x \in X_j\setminus \Sigma_{j}$. 
\end{itemize}
\end{enumerate}
\end{defi}
Note that this definition does not depend on the choice of $\omega$ by $X_{j} \Subset X$.

\begin{rem}
(1) In contrast to \cite[Definition 1.1]{GMY23}, 
our definition does not require  $h_{j,s}$ to be $C^{2}$ on $X_{j}$, 
allowing more flexibility in addressing Nakano positivity.
\\
(2) 
The inequality \( \ai\Theta_{h_{j,s}} \geq_\Nak (\theta-\lambda_{j,s}\omega) \otimes \Id_{E_{h_{j,s}}} \) in Condition (b) is valid on \( X_{j} \setminus \Sigma_{j,s} \). For illustration, we consider a smooth Hermitian metric \( h \) on \( E \) satisfying \( \ai\Theta_{h} \geq_\Nak \theta \otimes \Id_{E_{h}} \) at any point \( x \in X \setminus \Sigma \), where \( \Sigma \subset X \) is a subset. It can be easily verified that the condition \( \ai\Theta_{h} \geq_\Nak \theta \otimes \Id_{E_{h}} \) remains valid at any point in the closure of \( X \setminus \Sigma \).

\end{rem}

An alternative definition of Nakano positivity, as introduced in \cite{Ina22} and building upon \cite{Ber98, HI21} and most essentially \cite{DNWZ23}, 
employs the $\overline{\partial}$-equation with optimal $L^2$-estimates. 
As mentioned in Section \ref{Sec-1}, Theorem \ref{thm-seqimpliesL2} demonstrates that 
Definition \ref{def-singNak:sequenceNew} implies Definition \ref{def-singNak:l2global}.

\begin{defi}[{\cite[Definitions 1.1, 1.2]{Ina22}}]\label{def-singNak:l2global}
Consider the same situation as in Definition \ref{def-singNak:sequenceNew}. 

(1) We define \textit{the global $\theta$-Nakano positivity in the sense of $L^2$-estimates}, denoted by 
    \[
    \ai\Theta_h \geq^{L^2}_{\glo\Nak} \theta \otimes \Id_{E_h} \text{ on } X,
    \]
by the following condition:
For any data consisting of 
\begin{itemize}
\item a Stein coordinate $\Omega$ admitting a trivialization $E|_\Omega\cong \Omega\times \C^r$; 
\item a K\"ahler form $\omega_\Omega$ on $\Omega$; 
\item a smooth function $\psi$ on $\Omega$ such that $ \theta + \deldel \psi >0$; 
\item a positive integer $q$ with $1\leq q\leq n$; 
\item a $\overline{\partial} $-closed $v\in L^2_{n,q}(\Omega, E; \omega_\Omega, he^{-\psi})$ 
with 
$$\int_\Omega \langle B^{-1}_{\omega_\Omega,\psi,\theta\otimes \Id_E}v, v\rangle_{\omega_\Omega,h}e^{-\psi} dV_{\omega_\Omega}<\infty,$$ 
\end{itemize}
 there exists $u\in L^2_{n,q-1}(\Omega, E; \omega_\Omega, he^{-\psi})$ such that 
    \[
     \overline{\partial}  u=v \text{ and } 
    \int_\Omega |u|^2_{\omega_\Omega, h}e^{-\psi} dV_{\omega_\Omega}\leq \int_\Omega \langle B^{-1}_{\omega_\Omega,\psi,\theta\otimes \Id_E}v, v\rangle_{\omega_\Omega,h}e^{-\psi} dV_{\omega_\Omega},
    \]
    where $B_{\omega_\Omega,\psi,\theta\otimes \Id_E}=[\deldel\psi\otimes \Id_E+\theta\otimes \Id_E,\Lambda_{\omega_\Omega}]$. 
\smallskip 
\\
(2) 
We define \textit{the local $\theta$-Nakano positivity in the sense of $L^2$-estimates}, denoted by 

    \[
    \ai\Theta_h \geq^{L^2}_{\loc\Nak} \theta \otimes \Id_{E_h} \text{ on } X,
    \]
by the condition: For any point $x\in X$, there exists an open neighborhood $U$ of $x$ such that 
    \[
    \ai\Theta_h \geq^{L^2}_{\glo\Nak} \theta \otimes \Id_{E_h} \text{ on } U.
    \]
\end{defi}
In this paper, we focus on only global Nakano positivity. 
Therefore, for simplicity, we will write 
\[
\ai \Theta_h \geq^{L^2}_{\Nak} \theta\otimes \Id_{E_h} \text{ instead of  }
\ai \Theta_h \geq^{L^2}_{\glo\Nak} \theta \otimes \Id_{E_h}. 
\]

For smooth Hermitian metrics, it is a direct consequence of the definitions that Nakano positivity implies Griffiths positivity. 
This relation can be similarly extended to singular Hermitian metrics after a slight modification.

\begin{prop}\label{prop-NakimpliesGrif}
    Let $h$ be a singular Hermitian metric on $E$ such that 
    \[
    \ai\Theta_h \geq^s_{\Nak} 0 \text{ holds in the sense of Definition \ref{def-singNak:sequenceNew}}. 
    \]
Assume that $\Sigma_j=\cup_s\Sigma_{j,s}$ is a closed pluripolar set for $j \gg 0$.
Then $h$ is a.e.\,Griffiths semi-positive 
$($i.e.,\,the function $\log |u|_{h^{*}}$ is a.e.\,psh 
for any local holomorphic section $u$ of $E^{*}$$)$.  
\end{prop}

\begin{proof}
Consider a holomorphic section $u \in H^{0}(U, E^{*})$ defined on an open subset $U\Subset X$ and 
take a member $X_j$ from the open cover $\{X_{j}\}_{j=1}^{\infty}$ in Definition \ref{def-singNak:sequenceNew} such that $U \subset X_j $. 
Then, \cite[Proposition 2.3]{GMY22} indicates that $\log |u|_{h^*}$  is psh on  $U\setminus \Sigma_j$. 
Since $h^*$ is upper semi-continuous, the function $\log |u|^2_{h^*}$ is locally bounded above. 
Hence, there exists a psh function $v$ on $U$ such that  $v=\log |u|_{h^*}$ on $U\setminus \Sigma_j$.
(see \cite[Chapter I, Theorem 5.24]{Dem-book} for example). 
\end{proof}

\begin{rem}\label{rem-NakimpliesGrif}
Consider the following straightforward example, which shows that $h$ in the above proposition is not necessarily Griffiths semi-positive:
Let $E=\Delta\times \C$ be the trivial line bundle over the unit disc $X:=\Delta \subset \mathbb{C}^{1}$ 
and define the function $\varphi$  on $\Delta$  such that $\varphi(0)=1$ and $\varphi(z)=0$ for $z\in \Delta\setminus \{ 0\}$. 
We can see that the singular Hermitian metric $h:=e^{-\varphi}$ on $\Delta\times\C$ satisfies that 
\[
 \ai\Theta_h \geq^s_{\Nak} 0 \text{ in the sense of Definition \ref{def-singNak:sequenceNew}}
 \]
by choosing $X_j=\Delta$, $\Sigma_{j,s}=\{ 0\}$, and $h_{j,s}=e^{-\varphi}|_{X_{j}\setminus \{ 0\}}=1$ for all $j$ and $s$. 
However, the function $\varphi$ is not psh. 
\end{rem}

At the end of this section, we recall the lemma essential for the proof of Theorem \ref{thm-seqimpliesL2}.

\begin{lemm}[{\cite[Lemma 9.10]{GMY23}}]\label{lem-l2correcting}
    Let $X$ be a complete K\"ahler manifold with a $($not necessarily complete$)$ K\"ahler metric $\omega$, 
    and $(Q,h)$ be a vector bundle $Q$ with a smooth Hermitian metric $h$. 
Let   $\eta > 0$ and $g >0 $ be smooth functions on $X$ such that $(\eta+g)$ and $(\eta+g)^{-1}$ are bounded. 
Let $\lambda\geq 0$ be a bounded continuous function on $X$ such that $(B+\lambda I)$ is positive definite everywhere on $\wedge^{n,q}T^* X\otimes Q$, 
where $B:=[\eta\ai\Theta_Q-\deldel\eta-\ai g\partial\eta\wedge\overline{\partial} \eta, \Lambda_\omega]$.

Then, for a given form $v\in L^2_{n,q}(X, Q; \omega, h)$ 
with 
$$\text{
$\overline{\partial}  v=0$ and $\int_X \langle (B+\lambda I)^{-1}v, v\rangle_{\omega, h}dV_\omega<\infty$, 
 }$$  
   there exist an approximate solution $u\in L^2_{n,q-1}(X, Q; \omega, h)$ and a correcting term 
   $\tau \in L^2_{n,q}(X, Q; \omega, h)$ such that 
$\overline{\partial}  u+P_h (\sqrt{\lambda}\tau )=v$ and   \[
\int_X (\eta+g^{-1})^{-1}|u|^2_{\omega,h} dV_\omega+ \int_X |\tau|^2_{\omega,h}dV_\omega \leq \int_X \langle (B+\lambda I)^{-1}v, v\rangle_{\omega,h} dV_\omega, 
    \]
where $P_h\colon  L^2_{n,q}(X, Q; \omega, h)\to \Ker \overline{\partial} $ is the orthogonal projection.

\end{lemm}

\section{Proof of Theorem \ref{thm-seqimpliesL2}}

This section is devoted to proving Theorem \ref{thm-seqimpliesL2}.
Using the same notation as in Definition \ref{def-singNak:sequenceNew}, 
we first confirm preliminary results on $L^{2}$-spaces. 
For a section $u\in L^2_{n,q}(X_j\setminus\Sigma_{j,s}, E; \omega, h_{j,s})$, 
we define the {\textit{trivial extension}} $\widetilde{u}$ of $u$  by setting $\widetilde{u}=0$ on $\Sigma_{j,s}$. 
Similarly, we also define the trivial extension $\widetilde{h}_{j,s}$ of $h_{j,s}$.
Using these extensions, we obtain the isomorphism:
\begin{align}\label{eq-isom}
L^2_{n,q}(X_j\setminus\Sigma_{j,s}, E; \omega, h_{j,s}) \cong L^2_{n,q}(X_j, E; \omega, \widetilde{h}_{j,s}). 
\end{align}
Let $P_s$ and $P'_s$ be the first projections of the following orthogonal decomposition: 
\begin{align*}
    L^2_{n,q}(X_j\setminus\Sigma_{j,s}, E; \omega, h_{j,s})&= (\Ker \Dbar)_s \oplus (\Ker \Dbar)_s^{\perp},\\
    L^2_{n,q}(X_j, E; \omega, \widetilde{h}_{j,s}) &= (\Ker \Dbar)_{s'} \oplus (\Ker \Dbar)_{s'}^{\perp}.
\end{align*}
Here, the subscripts $s$ and $s'$ are only used to distinguish the spaces of $\Ker \Dbar$.
If  $u \in (\Ker \Dbar)_s $, 
then $\widetilde{u}\in (\Ker \Dbar)_{s'}$ by \cite[Lemma 6.9]{Dem82} 
and similarly if $v\in (\Ker\Dbar)_s^{\bot}$, then $\widetilde{v}\in (\Ker\Dbar)_{s'}^{\bot}$, 
which means that the above projections and $\Dbar$-operators are compatible with the trivial extensions. 
Consequently, via the isomorphism \eqref{eq-isom}, 
we can identify the $L^{2}$-space equipped with the projection 
on $X_{j} \setminus \Sigma_{j,s}$ with the corresponding $L^{2}$-space on $X_{j}$. 
Under this identification, for simplicity we write $h_{j,s}=\widetilde{h}_{j,s}$ and use the following notation
\[
L^2_{n,q}(X_j, E; \omega, h_{j,s}). 
\]
Since $\{ {h}_{j,s}\}_s$ increasingly converges to $h$ almost everywhere on $X_j$, 
we can see that 
\[
L^2_{n,q}(X_j, E; \omega, h) \subset L^2_{n,q}(X_j, E; \omega, h_{j,s})\subset L^2_{n,q}(X_j, E; \omega, h_{j,t})
 \text{ \quad for $t \leq s$}.
\]
This observation leads to the following lemma.

\begin{lemm}[{\cite[Lemma 9.1]{GMY23}}]\label{lem-GMY1}
Consider a sequence $\{ u_\ell \}_{\ell=1}^{\infty}$ that weakly converges to $u$ 
in $L^2_{n,q}(X_j, E; \omega, h_{j,s})$ and 
a sequence $\{ v_\ell\}_{\ell=1}^{\infty}$  of functions on $X_{j}$ that converges pointwise to $v$. 
If $\sup_{X_{j}} |v_\ell|$ is uniformly bounded in $\ell$, then $\{ v_\ell u_\ell\}_{\ell=1}^{\infty}$ weakly converges to $uv$ 
in $L^2_{n,q}(X_j, E; \omega, h_{j,s})$. 
\end{lemm}

By \cite[Section 9]{GMY23} (see also \cite{GMY22}), 
we have the linear isometries
\begin{align*}
    H_s&\colon L^2_{n,q}(X_j, E; \omega, h_{j,s}) \to L^2_{n,q}(X_j, E; \omega, h_{j,1}),\\
    H&\colon L^2_{n,q}(X_j, E; \omega, h) \to L^2_{n,q}(X_j, E; \omega, h_{j,1}).
\end{align*}
Precisely speaking, in \cite{GMY23}, the metric $h_{j,s}$ is assumed to be smooth on $X_{j}$ 
(not only $X_{j} \setminus \Sigma_{j,s}$). 
Nevertheless,  since $\Sigma_j$  is of measure zero, 
the linear isometries constructed for the $L^{2}$-spaces on $X_{j} \setminus \Sigma_{j,s}$ 
determines the above linear isometries. 
Using them, we obtain the following lemma:

\begin{lemm}\label{lem-GMY2}
$(${\cite[Lemma 9.9]{GMY23}, \cite[Lemma 2.2]{GMY22}, 
{\rm cf.}\,\cite[Proposition 3.4]{Mat22}, \cite[Proposition 3.5]{Mat18}}$)$. 
Consider a sequence $\{ f_s \}_{s=1}^{\infty}$ such that 
$f_s\in L^2_{n,q}(X_j, E; \omega, h_{j,s})$ and 
$\| f_s\|_s \leq C$ holds for a uniform constant $C$. 
Then, there exists $f \in L^2_{n,q}(X_j, E; \omega, h)$ and 
a subsequence of $\{ f_s\}_s$ $($for which we use the same notation$)$
with the following properties: 
\begin{itemize}
\item $\| f\|_{h} \leq C$; 
\item For $t \in \mathbb{Z}_{+}$, 
the sequence $\{ f_s\}_s$  weakly converges to $f$ in $L^2_{n,q}(X_j, E; \omega, h_{j,t})$; 
\item $\{P_s(f_s)\}_s$ weakly converges to $P(f_0)$ in $L^2_{n,q}(X_j, E; \omega, h_{j,t})$. 
\end{itemize}
Here $P\colon L^2_{n,q}(X_j, E; \omega, h)\to \Ker\Dbar$ is the orthogonal projection.
\end{lemm}

We now prove Theorem \ref{thm-seqimpliesL2}, building upon the above preliminary results.

\begin{proof}[Proof of Theorem \ref{thm-seqimpliesL2}]\label{sec-ProofTheorem1-2}

We freely use the notation in the proof as in Definitions \ref{def-singNak:sequenceNew} and \ref{def-singNak:l2global}.
Take an exhaustive Stein open cover $\{ \Omega_k\}_{k=1}^{\infty}$ of $\Omega$. 
For a fixed integer $k \in \mathbb{Z}_{+}$, 
we have $\Omega_k\Subset X_j$ for a sufficiently large $j\geq j_k$ by  $\Omega_k \Subset X$. 
Since we fix such a $j$, we omit the subscript $j$. 
By Definition  \ref{def-singNak:sequenceNew}, we have the following conditions on $\Omega_k$:
\begin{enumerate}
\item[(a)] Set $\Sigma:=\cup_{s=1}^{\infty}\Sigma_{s}$. 
Then, for any $x\in \Omega_{k} \setminus \Sigma$ and $e_x\in E_x$, we have 
$$
|e|_{h_{s}}\nearrow |e|_{h} \text{ as } s \nearrow \infty. 
$$ 
\item[(b)] There exist continuous functions $\lambda_{s}$ and $\lambda$ on $\overline{\Omega}_{k}$ 
such that 
\begin{itemize}
\item $0\leq \lambda_{s}\leq \lambda$ on $\Omega_{k}$; 
\item  $\lambda_{s}\to 0$ almost everywhere on $\Omega_{k}$; 
\item $\ai\Theta_{h_{s}}\geq_\Nak (\theta-\lambda_{s}\omega)\otimes \Id_{E_{h_{s}}}$ 
holds at any point $x \in \Omega_{k} \setminus \Sigma$. 
\end{itemize}
\end{enumerate}
    
Let us apply  Lemma \ref{lem-l2correcting} to 
the data $(X,\omega):=(\Omega_k\setminus \Sigma_{s},\omega_\Omega)$, $(Q,h):=(E,h_{s}e^{-\psi})$, 
$\eta=1$, $g=\delta$, $B_{s}:=B=[\ai\Theta_{h_{s}e^{-\psi}}, \Lambda_{\omega_\Omega}]$, 
and $(B_{s}+C_k\lambda_{s}I)$. 
Here $\delta$ is a positive number (which tends to zero later) and 
$C_k>0$ is a positive constant  such that $\omega\leq C_k\omega_\Omega$ on $\Omega_k$. 
Note that $\Omega_k\setminus \Sigma_{s}$ admits a complete K\"ahler metric by \cite[Theorem 1.5]{Dem82} and $\ai\Theta_{h_{s}}\geq_\Nak (\theta-C_k\lambda_{s}\omega_{\Omega})\otimes \Id_{E_{h_{s}}}$ holds
by Condition (b).
Then $B_{s}:=B=[\ai\Theta_{h_{s}e^{-\psi}}, \Lambda_{\omega_\Omega}]$ satisfies that 
    \begin{align*}
    \langle (B_{s}+C_k\lambda_{s}I)\bullet, \bullet\rangle_{\omega_\Omega, h_{s}} &\geq \langle [\deldel\psi\otimes \Id_E+\theta\otimes \Id_{E}, \Lambda_{\omega_\Omega}]\bullet, \bullet\rangle_{\omega_\Omega, h_{s}}\\
    &=\langle B_{\omega_\Omega,\psi, \theta\otimes \Id_{E}}\bullet, \bullet\rangle_{\omega_\Omega, h_{s}}.
    \end{align*}
Furthermore, we can easily see that 
    \begin{align*}
        \infty &> \int_{\Omega} \langle B^{-1}_{\omega_\Omega,\psi, \theta\otimes \Id_{E}}v, v\rangle_{\omega_\Omega, h}e^{-\psi}dV_{\omega_\Omega}\\
        & \geq \int_{\Omega_k \setminus \Sigma_{s}} \langle B^{-1}_{\omega_\Omega,\psi, \theta\otimes \Id_{E}} v, v\rangle_{\omega_\Omega, h}e^{-\psi}dV_{\omega_\Omega}\\
        &\geq \int_{\Omega_k \setminus \Sigma_{s}} \langle B^{-1}_{\omega_\Omega,\psi, \theta\otimes \Id_{E}} v, v\rangle_{\omega_\Omega, h_{s}}e^{-\psi}dV_{\omega_\Omega}\\
        &\geq \int_{\Omega_k \setminus \Sigma_{s}} \langle (B_{s}+C_k\lambda_{s}I)^{-1}v,v \rangle_{\omega_\Omega, h_{s}}e^{-\psi}dV_{\omega_\Omega}.
    \end{align*}
Lemma \ref{lem-l2correcting} shows that 
there exist an approximate solution $u_{k,s}\in L^2_{n,q-1}(\Omega_k\setminus \Sigma_{s}, E; \omega_\Omega, h_{s}e^{-\psi})$ and a correcting term $\tau_{k,s}\in L^2_{n,q}(\Omega_k\setminus \Sigma_{s}, E; \omega_\Omega, h_{s}e^{-\psi})$ such that 
\begin{equation}\label{eq:d-barkyokugenmae}
\overline{\partial}  u_{k,s}+P_{s}(\sqrt{C_k\lambda_{s}}\tau_{k,s})=v \text{ on $\Omega_k\setminus \Sigma_{s}$ }
\text{ and }
\end{equation}
\begin{align}\label{eq:L2kyokugenmae}
\int_{\Omega_k\setminus \Sigma_{s}}\frac{|u_{k,s}|^2_{\omega_\Omega,h_{s}}e^{-\psi}}{(1+\frac{1}{\delta})}dV_{\omega_\Omega}&+\int_{\Omega_k\setminus\Sigma_{s}} |\tau_{k,s}|^2_{\omega_\Omega,h_{s}}e^{-\psi}dV_{\omega_\Omega}
\\ & \leq \int_{\Omega} \langle B^{-1}_{\omega_\Omega,\psi, \theta\otimes \Id_{E}}v, v\rangle_{\omega_\Omega, h}e^{-\psi}dV_{\omega_\Omega}. \notag
\end{align}
Note that $C_k\lambda_{s}\geq 0$ is a bounded continuous function on $\Omega_k$. 

The  $L^{2}$-sections $u_{k,s}$, $\tau_{k,s}$ and Equation \eqref{eq:d-barkyokugenmae}
are a prior considered on $\Omega_k\setminus \Sigma_{s}$. 
Nevertheless, as mentioned at the beginning of Section \ref{sec-ProofTheorem1-2}, 
we may assume that $u_{k,s}$ and $\tau_{k,s}$ are $L^{2}$-sections on $\Omega_{k}$ 
and satisfies Equation \eqref{eq:d-barkyokugenmae} on $\Omega_{k}$. 

Since $\| \sqrt{C_k\lambda_s}\tau_{k,s}\|_s$ and $\| {\tau}_{k,s} \|_{s}$ are bounded uniformly in $s$, 
we can find a subsequence (for which we use the same notation) 
such that $\{ \sqrt{C_k\lambda_s}\tau_{k,s}\}_s$, $\{ P_s(\sqrt{C_k\lambda_s}\tau_{k,s})\}_s$, 
and $\{ {\tau}_{k,s} \}_{s}$ are weakly convergent  
in $L^2_{n,q}(\Omega_k, E; \omega_\Omega, h_{t}e^{-\psi})$ for a fixed $t\in \Z_{\geq 0}$
(see Lemma \ref{lem-GMY2}). 
Since $0\leq \lambda_s \leq \lambda$ and $\lambda_{s}$ converges pointwise to $0$, 
we see that $P_{s}(\sqrt{C_k\lambda_{s}}\tau_{k,s})\to P(0)=0$ weakly as $s\to \infty$ in $L^2_{n,q}(\Omega_k, E; \omega_\Omega, h_{t}e^{-\psi})$ by Lemma \ref{lem-GMY1}. 
Similarly, by applying Lemma \ref{lem-GMY2} to $\{ {u}_{k,s}\}_s$, 
we may assume that for some (and by the diagonal argument for any) $t\in \Z_{\geq 0}$ $\{ {u}_{k,s} \}_{s}$ weakly converges to $u_{k}$ 
in $L^2_{n,q-1}(\Omega_k, E; \omega_\Omega, h_{t}e^{-\psi})$. 
Therefore, by taking $s\to \infty$ in (\ref{eq:d-barkyokugenmae}), we have that 
\begin{equation}\label{eq:d-bar}
\overline{\partial}  u_{k}=v \text{ on $\Omega_k$ \text{ and } }
\end{equation}
\begin{align}\label{eq:L2}
\int_{\Omega_k }\frac{|u_{k}|^2_{\omega_\Omega,h_{t}}e^{-\psi}}{(1+\frac{1}{\delta})}dV_{\omega_\Omega} \leq \int_{\Omega} \langle B^{-1}_{\omega_\Omega,\psi, \theta\otimes \Id_{E}}v, v\rangle_{\omega_\Omega, h}e^{-\psi}dV_{\omega_\Omega}.
\end{align}
The monotone convergence theorem shows that 
\begin{equation*}
\int_{\Omega_k} \frac{| {u}_{k}|^2_{\omega_\Omega,h} e^{-\psi}}{(1+\frac{1}{\delta})}dV_{\omega_\Omega} \leq \int_{\Omega} \langle B^{-1}_{\omega_\Omega,\psi, \theta\otimes \Id_{E}}v, v\rangle_{\omega_\Omega, h}e^{-\psi}dV_{\omega_\Omega}.
\end{equation*}

Let us regard $u_{k}$ (define on $\Omega_{k}$) as an $L^{^{2}}$-section on $\Omega$ using the zero extension. 
Then, in the same way as in Lemma \ref{lem-GMY2} (cf.\,the proof of \cite[Proposition 3.4]{Mat22}), 
we can find a subsequence of $\{ u_{k}\}_k$ that 
weakly converges to some $u$ satisfying that 
\begin{equation*}
    \int_{\Omega} \frac{|u|^2_{\omega_\Omega,h} e^{-\psi}}{(1+\frac{1}{\delta})}dV_{\omega_\Omega} \leq \int_{\Omega} \langle B^{-1}_{\omega_\Omega,\psi, \theta\otimes\Id_{E}}v, v\rangle_{\omega_\Omega, h}e^{-\psi}dV_{\omega_\Omega}. 
\end{equation*}
By taking $\liminf_{\delta\to\infty}$ and using Fatou's lemma, we obtain that 
\begin{equation*}
    \int_{\Omega} |u|^2_{\omega_\Omega,h} e^{-\psi} \leq \int_{\Omega} \langle B^{-1}_{\omega_\Omega,\psi, \theta\otimes\Id_{E}}v, v\rangle_{\omega_\Omega, h}e^{-\psi}dV_{\omega_\Omega}.
\end{equation*}
We also see that $\overline{\partial}  u=v$ on $\Omega$, which finishes the proof. 
\end{proof}

As an application of Theorem \ref{thm-seqimpliesL2}, 
we derive the following corollary, which extends the results of \cite[Proposition 6.1]{Ina22}. 
Notably, we see that if a smooth Hermitian metric $h$ satisfies the conditions of Definition \ref{def-singNak:sequenceNew}, 
then $h$ is Nakano positive in the usual sense by \cite[Theorem 1.1]{DNWZ23}. 
This conclusion is significantly non-trivial.

\begin{cor}\label{cor-increasingNakano}
Let $E$ be a vector bundle on a complex manifold $X$ and $h$ be a singular Hermitian metric on $E$. 
Assume that there exists a sequence $\{ h_\nu\}_{\nu=1}^\infty$ of 
smooth Nakano semi-positive metrics $h_\nu$ on $E|_{X\setminus \Sigma_\nu}$ 
such that $\{ h_\nu\}_{\nu=1}^\infty$ increasing to $h$ pointwise on $X\setminus \cup_{\nu=1}^\infty \Sigma_{\nu}$, 
where $\Sigma_\nu$ is a $($proper closed$)$ analytic subset of $X$. 
Then, it follows that 
    \[
    \ai\Theta_{h} \geq^s_{\Nak} 0 \text{ \ and \ } \ai\Theta_{h} \geq^{L^2}_{\Nak} 0.
    \]
In particular, if $h$ is smooth, then $h$ is a Nakano semi-positive metric in the usual sense.
\end{cor}

\section{Applications of Theorem \ref{thm-seqimpliesL2}}
This section presents several applications of Theorem \ref{thm-seqimpliesL2} and its proof.
We first establish the following theorem, which is a generalization of Nadel's coherence theorem for multiplier ideal sheaves.

\begin{theo}\label{thm-coherence}
Let $E$ be a vector bundle on a complex manifold $X$. 
   Let $h$ be a singular Hermitian metric on $E$ such that 
\[
\ai\Theta_h \geq^s_{\Nak} \theta \otimes \Id_{E_h} \text{ holds in the sense of Definition \ref{def-singNak:sequenceNew}}
\]
for some continuous real $(1,1)$-form  $\theta$ on $X$. 
Then, the sheaf $\mathcal{E}(h)$ of locally $L^2$-integrable sections of $E$ with respect to $h$ 
$($see \cite[Definition 2.3.1]{deC98}$)$ 
is coherent. 
\end{theo}

\begin{proof}\label{proof-thm-coherence}
By  Theorem \ref{thm-seqimpliesL2}, we see that 
    \[
    \ai\Theta_h \geq^{L^2}_{\Nak} \theta \otimes \Id_{E_h}.
    \]
Since the coherence is a local property, 
we consider local Stein open coordinate $U\Subset V \Subset X$ 
and a smooth function $\varphi$ on $V$ 
such that $\omega:=\deldel \varphi$ is a K\"ahler form on $V$. 
By taking a sufficiently large $c>0$, we have that 
    \[
    \ai\Theta_{he^{-c \varphi}} \geq^{L^2}_{\Nak} 0 \text{ on } U.
    \]
Then, observing that $\mathcal{E}(h)=\mathcal{E}(he^{-c\varphi})$, 
we deduce the coherence from \cite[Theorem 1.4]{HI21} or \cite[Proposition 4.4]{Ina22}.

\end{proof}

By invoking the proof of Theorem \ref{thm-seqimpliesL2}, we can solve the $\Dbar$-equation with $L^2$-estimates. 

\begin{theo}\label{thm-weaklyL2}
Let $E$ be a vector bundle on a weakly pseudoconvex K\"ahler manifold $(X, \omega)$. 
Let $f\colon X\to \R$ be a positive continuous on $X$ and $h$ be a singular Hermitian metric on $E$ satisfying 
    \[
    \ai\Theta_h \geq^s_{\Nak} f\omega\otimes \Id_{E_h} \text{    in the sense of Definition \ref{def-singNak:sequenceNew}.}
    \]
    Then, for any $\overline{\partial} $-closed $v\in L^2_{n,q}(X,E;\omega, h)$ with 
    \[
    \int_X \frac{|v|^2_{\omega,h}}{f} dV_\omega <\infty,
    \]
     there exists $u\in L^2_{n,q-1}(X,E;\omega, h)$ such that 
     \[
     \overline{\partial}  u=v \text{ and } \int_X |u|^2_{\omega,h}dV_\omega \leq \int_X \frac{|v|^2_{\omega,h}}{qf} dV_\omega.
     \]
\end{theo}

\begin{proof}
The proof is essentially the same as in the proof of Theorem \ref{thm-seqimpliesL2}; thus, we present only a sketch of the proof.
Take an open cover $\{X_{j}\}_{j=1}^{\infty}$ of $X$ 
such that $X_{j} \Subset X_{j+1} \Subset X$ holds and $X_j$ is a weakly pseudoconvex K\"ahler manifold. 
Define $B_{j,s}:=[\ai\Theta_{h_{j,s}}, \Lambda_\omega]$.  
By $\ai\Theta_{h_{j,s}}\geq_{\Nak}(f-\lambda_{j,s})\omega\otimes\Id_{E_{h_{j,s}}}$ on $X_j\setminus \Sigma_{j,s}$, 
we see that 
    \begin{align*}
         \langle (B_{j,s}+\lambda_{j,s}I)\bullet, \bullet\rangle_{\omega, h_{j,s}} &\geq \langle [f\omega\otimes\Id_E, \Lambda_{\omega}]\bullet, \bullet\rangle_{\omega, h_{j,s}}=qf|\bullet|^2_{\omega, h_{j,s}}. 
    \end{align*}
This inequality implies that 
    \begin{align*}
        \infty >& \int_X \frac{|v|^2_{\omega, h}}{qf} dV_\omega \\
        \geq&  \int_{X_j\setminus \Sigma_{j,s}} \frac{|v|^2_{\omega, h}}{qf} dV_\omega \\
        \geq& \int_{X_j\setminus \Sigma_{j,s}} \langle (B_{j,s}+\lambda_{j,s}I)^{-1}v, v\rangle_{\omega, h_{j,s}}  dV_\omega. 
    \end{align*}
Then, in the same way as in the proof of Theorem \ref{thm-seqimpliesL2}, 
by applying Lemma \ref{lem-l2correcting} and taking the limit as $s\to \infty, j\to \infty$, 
we obtain the desired conclusion. 
\end{proof}

As an application of Theorem \ref{thm-weaklyL2}, we obtain Theorem \ref{mainthm-vanishing}. 

\begin{proof}[Proof of Theorem \ref{mainthm-vanishing}]
Take an arbitrary $\overline{\partial} $-closed and locally $L^{2}$-integrable 
$v \in L^{2, \rm{loc}}_{n,q}(X,E;\omega', h)$. 
Note that the space $L^{2, \rm{loc}}_{n,q}(X,E;\omega', h)$ does not depend on the choice of $\omega'$. 
By the standard argument using the De Rham-Weil isomorphism, 
it is sufficient to find $u\in L^{2, \rm{loc}}_{n,q-1}(X,E;\omega', h)$ satisfying the $\Dbar$-equation $\Dbar u =v$. 
Take a smooth function $f>0$ on $X$ such that $\omega' \geq f \omega$ holds on $X$. 
Since $X$ is weakly pseudoconvex, we can take an exhaustive smooth psh function $\psi$ on $X$. 
Since $\psi$ is exhaustive and $v$ is locally $L^{2}$-integrable, 
there exists an increasing convex function $\chi\colon \mathbb{R} \to \mathbb{R}$ 
such that $v/\sqrt{f}$ is $L^{2}$-integrable on $X$ with respect to $he^{-\chi \circ \psi}$ and $\omega$. 
Since $\chi \circ \psi$ is a psh function on $X$, we can easily see that 
    \[
    \ai\Theta_{he^{-\chi \circ \psi}} \geq^s_{\Nak} f\omega\otimes \Id_{E_h} \text{    in the sense of Definition \ref{def-singNak:sequenceNew}.}
    \]
Then, by Theorem \ref{thm-weaklyL2}, there exists 
$u\in L^{2}_{n,q-1}(X,E;\omega, he^{-\chi \circ \psi})\subset L^{2, \rm{loc}}_{n,q-1}
(X,E;\omega', h)$ satisfying the $\Dbar$-equation $\Dbar u =v$. 
\end{proof}

When $X$ is a projective manifold, 
the theorem immediately follows from Theorem \ref{thm-seqimpliesL2} and \cite[Theorem 1.5]{Ina22}. 
In \cite{Ina22}, it was shown that if $\ai\Theta_h \geq^{L^2}_{\Nak} \delta\omega_X \otimes \Id_{E_h} $ for a projective manifold $(X, \omega_X)$, it holds that $H^q(X, K_X\otimes \mathcal{E}(h))=0$ for $q>0$. 
(The definition in \cite{Ina22} assumes that $h$ is Griffiths semi-positive in the sense of singular Hermitian metrics, 
but this condition is not necessary in the proof of the vanishing theorem; it suffices for $h^* $ to be upper semi-continuous.)

At the end of this section, let us observe that \cite[Theorem 1.3]{Iwa21} follows from Theorem \ref{mainthm-vanishing}.

\begin{theo}[{\cite[Theorem 1.3]{Iwa21}}]\label{thm-Iwaivanishing}
    Let $(X,\omega) $ be a compact K\"ahler manifold and $(E,h) $ be a holomorphic vector bundle on $X$ with a singular Hermitian metric. We assume the following conditions. 
    \begin{enumerate}
        \item There exists a proper analytic subset $Z$ such that $h$ is smooth on $X\setminus Z$.
        \item  $he^{-\zeta}$ is a positively curved singular Hermitian metric on $E$ for some continuous function $\zeta$. 
        \item  There exists a positive number $\varepsilon>0$ such that $\ai \Theta_{h}-\varepsilon\omega \otimes \Id_{E_{h}}\geq 0 $ on $X\setminus Z $ in the sense of Nakano. 
    \end{enumerate}
    Then $H^q(X, K_X\otimes \mathcal{E}(h))=0 $ holds for any $q>0$.  
\end{theo}

The above setting is a special case of our formulation. 
Indeed, if we set $X_j=X, \Sigma_{j,s}=Z, h_{j,s}=h|_{X\setminus Z} $ for all $j,s$, 
then $h_{X\setminus Z}$ is a singular Hermitian metric on $E$ satisfying 
\begin{equation*}
    \ai\Theta_{h} \geq^s_{\Nak} \varepsilon \omega \otimes \Id_{E_h}. 
\end{equation*}
Then, the desired vanishing follows from Theorem \ref{mainthm-vanishing}.

\section{On the positivity of direct image sheaves}

In this section, we investigate  conditions under which the direct image sheaf $f_{*}(\mathcal{O}_{X}(K_{X/Y}+L))$ 
satisfies  Nakano positivity as defined in Definition \ref{def-singNak:sequenceNew}. 
Here $f \colon  X \to Y$ is a projective fibration with a line bundle $L$ on $X$. 
Since direct image sheaves are not always locally free, we extend the notion of Nakano positivity 
from vector bundles to torsion-free sheaves. 
In this paper, we identify locally free sheaves with vector bundles.

\begin{defi}[cf.\,Definition \ref{def-singNak:sequenceNew}]\label{def-singNak:sequenceNew-sheaf}
Let $\omega$ be a Hermitian form and $\theta$ be a continuous  real $(1,1)$-form on a complex manifold $X$. 
Let $\mathcal{E}$ be a torsion-free sheaf on $X$ and 
$X_{\mathcal{E}}$ denote the largest Zariski open set where $\mathcal{E}$ is locally free. 
Consider a singular Hermitian metric $h$ on $\mathcal{E}$ 
(i.e.,\,a singular Hermitian metric on the vector bundle $\mathcal{E}|_{X_{\mathcal{E}}}$). 
Assume that the function $|u|^2_{h^*}$ is upper semi-continuous for any local holomorphic section $u$ of $(\mathcal{E}|_{X_{\mathcal{E}}})^*$. 
We define \textit{the $\theta$-Nakano positivity in the sense of approximations}
by the existence of the data 
$(\{ X_j\}_{j=1}^{\infty}, \{ \Sigma_{j,s}\}_{j,s=1}^{\infty}, \{ h_{j,s}\}_{j,s=1}^{\infty})$ 
satisfying the same conditions as in Definition \ref{def-singNak:sequenceNew}, 
with the substitution of $E$ by $\mathcal{E}$ 
and the addition of the condition that $X\setminus X_{\mathcal{E}} \subseteq \Sigma_{j,s}$.
Note that $\{ X_j\}_{j=1}^{\infty}$ is an open cover of $X$ (not only $X_{\mathcal{E}}$). 
\end{defi}

In this section, we first present a proof of Theorem \ref{thm-direct}.
The proof is based on a technical combination of 
Demailly's approximation theorem (see \cite[Main Theorem 1.1]{Dem92}) and 
the Nakano positivity of Narasimhan-Simha metrics 
for a smooth fibration  $f \colon X \to Y$ with a semi-positive line bundle $L$ 
(see \cite{Ber09} and \cite[Theorem 1.6]{DNWZ23}).

\begin{proof}[Proof of Theorem \ref{thm-direct}]
Initially, we consider the case where $Y$ (and thus $X$) is compact. 
Fix an open cover $\{U_{i}\}_{i \in I}$ of $X$ such that $L|_{U_{i}}$ admits a trivialization. 
Then, since $g$ is a smooth metric, we may assume that $g \leq 1$ by scaling $g$ with an appropriate constant. 
Similarly,  since the curvature current of $h$ is bounded below,  we may assume that $h \geq 2$.
Here, the above inequalities are interpreted under the identification of 
the functions on $U_{i}$ and the metrics on $L|_{U_{i}}$  via the fixed trivialization of $L|_{U_{i}}$. 
Throughout the proof, we suppose that  $\e >0 $ and $ \delta >0$ are sufficiently small rational numbers.
By Demailly's approximation theorem (see \cite[Main Theorem 1.1]{Dem92}), 
we can find singular Hermitian metrics $\{h_{\delta}\}_{\delta > 0}$ on $L$ satisfying the following conditions: 
\begin{equation}\label{eq-app}
  \begin{split}
&\text{$\bullet$ $h_{\delta}$ has analytic singularities along an analytic subset $Z_{\delta} \subset X$;} \\
&\text{$\bullet$ $\ai \Theta_{h_{\delta}} \geq f^*\theta -\delta \omega_{X}$ holds;} \\
&\text{$\bullet$ $h_{\delta} \nearrow h$ as $\delta \searrow 0$.} 
  \end{split}
\end{equation}
Note that $h_{\delta} \geq 1 $ holds since $h_{\delta}$ is lower semi-continuous. 
Additionally, the analytic subset $Z_{\delta} \subset X$ is contained in $\{x \in X \mid \nu(h, x)>0\}$, 
indicating that $Z_{\delta}$  is not dominant over $Y$ by assumption. 
We fix a sufficiently large $c>0$ with $c\omega_Y\geq \theta$. 
Define $\lambda_{\e}$ and $\lambda$  as follows:
$$
\text{
$\lambda:=(1+c) \geq \lambda_{\e}:=\e (1 + c) $ 
ensuring that $\lambda_{\e}  \omega_{Y} \geq \e( \omega_{Y}+\theta)$ holds on $Y$. 
}
$$

Consider  the singular Hermitian metric $H_{\e}$ on $L$ defined by 
$$
H_{\e}:=g^{\e} \cdot (h_{\delta(\e)})^{1 - \e}, \text{ where } \delta(\e):=\dfrac{\e} {1-\e}. 
$$
We now confirm that $H_{\e}$ satisfies the following conditions: 
\begin{itemize}
\item[(a)]  $\ai \Theta_{H_{\e}} \geq f^{*}\big(\theta - \lambda_{\e} \omega_{Y} \big) $;
\item[(b)] $\mathcal{I}(H_{\e})=\mathcal{O}_{X}$ on $X \setminus Z_{\delta(\e)}$;
\item[(c)] $H_{\e} \nearrow h$ as $\e \searrow 0$.   
\end{itemize}
Condition (a) is derived from the straightforward computation: 
\begin{align*}
\ai \Theta_{H_{\e}} &= \e \ai\Theta_{g} +  (1-\e) \ai\Theta_{h_{\delta(\e)}} \\
& \geq - \e f^{*} \omega_{Y} + (\e - (1-\e) \delta(\e) )\omega_{X} + (1-\e) f^{*} \theta \\
& \geq  f^{*} \big( \theta - \lambda_{\e} \omega_{Y}\big)
\end{align*}
Condition (b) is obvious since $H_{\e}$ is smooth on $X \setminus Z_{\delta(\e)}$. 
Furthermore, by the definition of $H_{\e}$, it is also clear that $H_{\e}$ converges to $h$ at every point in $X$. 
The remaining task is to check the monotonicity of $H_{\e}$. 
For $\e_{1} \leq \e_{2}$, by noting that $\delta(\e_{1}) \leq \delta(\e_{2})$, 
we find $h_{\delta(\e_{1})} \geq h_{\delta(\e_{2})}$, which implies that
\begin{align*}
\frac{H_{\e_{1}}} {H_{\e_{2}}}&=g^{\e_{1} - \e_{2}} \cdot \frac{h_{\delta(\e_{1})}^{1-\e_{1}}}{h_{\delta(\e_{2})}^{1-\e_{2}}} \\
&\geq g^{\e_{1} - \e_{2}} \cdot h_{\delta(\e_{1})}^{\e_{2} - \e_{1}} 
\end{align*}
Hence, the desired monotonicity follows from $g \leq 1$ and $h_{\delta} \geq 1$.

Consider the natural injective sheaf morphism:  
$$
f_{*}( \mathcal{O}_{X}(K_{X/Y}+L)\otimes \mathcal{I}(H_{\e})) \to f_{*}(\mathcal{O}_{X}(K_{X/Y}+L))=:\mathcal{E}. 
$$
Note that this morphism is an isomorphism on $Y \setminus f(Z_{\delta(\e)})$ by Condition (b). 
Let $G_{\e}$ (resp.\,$G$) be the Narasimhan-Simha  metric on $\mathcal{E}$ 
induced by $H_{\e}$ (resp.\,$h$) and the above morphism (see \cite{HPS18, PT18}). 
We aim to show that $G_{\e}$ and $G$ satisfy the conditions of Definition \ref{def-singNak:sequenceNew-sheaf}. 
Take an analytic subset $V \subset Y$ such that $f \colon  X \to Y$ be a smooth fibration over $Y \setminus V$ 
and 
$$
f_{*}(\mathcal{O}_{X}(K_{X/Y}+L)) \otimes \mathcal{O}_{Y,y}/\mathfrak{m}_{y} 
\cong H^{0}(X_{y}, \mathcal{O}_{X_{y}}(K_{X_{y}} \otimes L|_{X_{y}}))
$$
for a point $ y \in Y \setminus V$. 
Here $\mathfrak{m}_{y}  \subset \mathcal{O}_{Y,y}$ is the maximal ideal of the stalk $\mathcal{O}_{Y,y}$ 
and $X_{y}$ denotes the fiber of $f \colon X \to Y $ at $y$. 
Set $\Sigma_{\e}:= f(Z_{\delta(\e)}) \cup V$ and $\Sigma:=\cup_{0<\e \in \mathbb{Q}_{+}} \Sigma_{\e}$. 
The metric $G_{\e}$ is a smooth Hermitian metric on $\mathcal{E}|_{Y \setminus \Sigma_{\e}}$. 
Moreover, by \cite{Ber09} and \cite[Theorem 1.6]{DNWZ23}, 
Condition (a) indicates that $\ai \Theta_{G_{\e}}\geq_\Nak (\theta - \lambda_{\e} \omega_{Y} )\otimes \Id_{G_{\e}}$ on $Y \setminus \Sigma_{\e}$. 
Meanwhile, by the construction of Narasimhan-Simha metrics, for a (local) section $s$ of $\mathcal{E}$, 
the metrics $G_{\e}$ and $G$  at $y \in Y \setminus \Sigma$ can be expressed as the fiber integral: 
$$
|s|^{2}_{G_{\e}}=\int_{X_{y}} |s  |_{X_{y}} |^{2}_{H_{\e}} \quad \text{ and } \quad |s|^{2}_{G}=\int_{X_{y}} |s  |_{X_{y}} |^{2}_{h}. 
$$
Hence, Condition (c) indicates that $G_{\e} \nearrow G$ holds at any point in $Y\setminus \Sigma$.
This completes the proof in the case where $Y$ is compact.

\smallskip

In the case where $Y$ is non-compact, we encounter a new problem: 
the inequalities $g \leq 1$ and $h \geq 2$ are not satisfied even after scaling $g$ and $h$ with a constant. 
We can find such a constant on a relatively compact subset $X_{j} \Subset X$, 
but the scaling with this constant alters the target $h$ that 
requires approximation, and this alteration depends on each $j$. 
Therefore, some technical modifications are required to address this issue.
The strategy of the proof is to apply the arguments in the first half for the modified metrics $\widetilde g$ and $\widetilde h$, 
which are defined below. 

Since $f \colon  X \to Y$ is a proper fibration, 
we can find smooth functions $\varphi_{g}$ and $\varphi_{h}$ on $Y$
such that 
$$
\widetilde g:=g \cdot e^{-f^{*}\varphi_{g}} \leq 1 \quad \text{ and }\quad \widetilde h:=h\cdot e^{-f^{*}\varphi_{h}} \geq 2.
$$
Then, we can easily check that 
$$
\sqrt{-1}\Theta_{\widetilde h} \geq f^{*} (\theta + \deldel \varphi_{h}) \quad \text{ and } \quad 
\sqrt{-1}\Theta_{\widetilde g} + f^{*} (\omega_{Y} - \deldel \varphi_{g})  \geq \omega_{X}. 
$$
 Take an open cover $\{Y_j\}_{j=1}^{\infty}$  of $Y$ 
such that $Y_j \Subset Y_{j+1} \Subset Y$ and define $X_{j}:=f^{-1}(Y_{j})$. 
Note that the open cover $\{X_j\}_{j=1}^{\infty}$ of $X$ clearly satisfies $X_j \Subset X_{j+1} \Subset X$. 
For each $j \in \mathbb{Z}_{+}$, 
by applying  \cite[Main Theorem 1.1]{Dem92} to $h|_{X_{j}}$, 
we can find singular Hermitian metrics $\{ h_{j, \delta}\}_{\delta>0}$ on $L|_{X_{j}}$ satisfying \eqref{eq-app} on $X_{j}$. 
Precisely, \cite[Main Theorem 1.1]{Dem92} assumes that $X$ is compact, 
but this theorem remains valid for the relatively compact subset $X_{j} \Subset X$ 
(see \cite[Theorem 2.9]{Mat22} for the detailed argument). 
We fix a sufficiently large $c_{j}$ such that 
$$c_j \omega_Y \geq \theta + \deldel(\varphi_h-\varphi_g) \text{ on } Y_j. $$
We define $\lambda_{j,\e}:=\e(1+c_j) \leq (1+c_j)=:\lambda_j$ so that 
\[
\lambda_{j,\e} \omega_Y \geq \e (\omega_Y + \theta + \deldel(\varphi_h-\varphi_g)) \text{ on }Y_j. 
\]
Consider the new metric ${H}_{j,\e}$ defined by
\[
{H}_{j,\e}=\widetilde{g}^\e\cdot (h_{j,\delta(\e)})^{1-\e}e^{f^*\varphi_h}
\]
Then, Conditions (a), (b), (c) on $X_{j}$ can be easily confirmed by the same argument as in the first half.

Let $G_{j,\e}$ (resp.\,$G_j$) be the singular Hermitian metric on $f_{*}(\mathcal{O}_{X}(K_{X/Y}+L))|_{Y_{j}}$ 
induced by $H_{j,\e}$ (resp.\,$h|_{X_{j}}$). 
Define $\Sigma_{j}$ (resp.\,$\Sigma_{j,\e}$) in the same way as $\Sigma$ (resp.\,$\Sigma_{\e}$). 
Then, by the same argument as in the first half, 
we deduce that 
$\ai \Theta_{{G}_{j,\e}}\geq_\Nak (\theta - \lambda_{j,\e} \omega_{Y} )\otimes \Id$ on $Y_j\setminus \Sigma_{j,\e}$ and 
$G_{j,\e}\nearrow G_{j}$ on $Y_j\setminus \Sigma_{j}$, which finishes the proof. 
\end{proof}

At the end of this paper, as an application of Theorems \ref{thm-direct} and \ref{mainthm-vanishing}, 
we establish Theorem \ref{thm-vanishing-vec}. 
This theorem generalizes \cite[Theorem 1.5]{Iwa21}, 
relaxing its assumptions and extending its applicability to weakly pseudoconvex K\"ahler manifolds.
We emphasize that a generalization to weakly pseudoconvex manifolds is natural from the viewpoint of several complex variables.

\begin{proof}[Proof of Theorem \ref{thm-vanishing-vec}]
Set $f \colon Z:=\mathbb{P}(E) \to X$ and $L:=\mathcal{O}_{\mathbb{P}(E)}(1)$. 
Consider the smooth Hermitian metric $g$ on $L$ induced by a smooth Hermitian metric on $E$ and $f^{*}E \to L$. 
It is easy to see that  
$\sqrt{-1}\Theta_{g} + f^{*} \omega_{X} \geq \omega_{Z}$ holds 
for some Hermitian forms. 
Since $\sqrt{-1}\Theta_{h_{q}} $ is a K\"ahler current and $f \colon Z \to X$ is a proper fibration, 
there exists a Hermitian form $\omega'$ on $X$ such that 
$\sqrt{-1}\Theta_{h_{q}} \geq f^{*} \omega'$. 
Let $G$ be the Narasimhan-Simha metric $G$ on $S^{m} E \otimes \det E$ 
induced by the formula 
$$
S^{m} E \otimes \det E \cong  f_{*}\mathcal{O}_{Z}(mL + f^{*}\det E)
= f_{*}\mathcal{O}_{Z}(K_{Z/X} + (r+m+1) L),
$$
where $r:=\rank E$. 
Theorem \ref{thm-direct} shows that 
$\ai\Theta_G \geq^{s}_\Nak    \omega' \otimes \Id_{G}$   on $Y$ 
holds in the sense of Definition \ref{def-singNak:sequenceNew}. 
Meanwhile, we have $G=S^{m}h \otimes \det h$ (see \cite[25]{HPS18} for example). 
Thus, the desired vanishing theorem directly follows from  Theorem \ref{mainthm-vanishing}. 
\end{proof}

\end{document}